\documentclass[10pt]{amsart}

\usepackage{amssymb,amsthm,amsmath,enumerate}
\usepackage[numbers,sort&compress]{natbib}
\usepackage{color}
\usepackage{graphicx}
\usepackage{amssymb,amsthm,amsmath}
\usepackage[numbers,sort&compress]{natbib}
\usepackage{color}
\usepackage{graphicx}

\title[ ]{ Criteria for  embedded  eigenvalues  for  discrete Schr\"odinger operators }

\author{Wencai Liu}
\address[Wencai Liu]{Department of Mathematics, University of California, Irvine, California 92697-3875, USA}\email{liuwencai1226@gmail.com}

\theoremstyle{plain}
\newtheorem{theorem}{Theorem}[section]
\newtheorem{corollary}[theorem]{Corollary}
\newtheorem{lemma}[theorem]{Lemma}
\newtheorem{proposition}[theorem]{Proposition}

\newcommand{\R}{\mathbb{R}}
\newcommand{\Z}{\mathbb{Z}}
\newcommand{\N}{\mathbb{N}}
\theoremstyle{definition}

\newtheorem{remark}[theorem]{Remark}

\begin{document}


\begin{abstract}
In this paper,  we consider   discrete  Schr\"odinger operators of the form,
\begin{equation*}
     (Hu)(n)=  u({n+1})+u({n-1})+V(n)u(n).
\end{equation*}
We view $H$  as a  perturbation of  the free operator $H_0$, where $(H_0u)(n)=  u({n+1})+u({n-1})$.
For $H_0$ (no perturbation), $\sigma_{\rm ess}(H_0)=\sigma_{\rm ac}(H)=[-2,2]$ and  $H_0$ does not have eigenvalues embedded into $(-2,2)$.
 It is an interesting  and important problem to identify the perturbation such that  the operator $H_0+V$ has one eigenvalue (finitely many eigenvalues or countable eigenvalues)
 embedded into $(-2,2)$.
 We introduce the {\it almost sign  type potential } and develop  the Pr\"ufer transformation to   address this problem,
 which leads to the   following five  results.

 \begin{description}
   \item[1] We obtain the sharp spectral  transition for  the existence of   irrational type eigenvalues  or rational type    eigenvalues  with even denominator.
   \item[2] Suppose  $\limsup_{n\to \infty} n|V(n)|=a<\infty.$
  We obtain a lower/upper  bound of $a$ such that $H_0+V$ has  one rational type    eigenvalue with odd denominator.
   \item[3] We  obtain the asymptotical behavior   of  embedded  eigenvalues around the boundaries of $(-2,2)$.
   \item [4]Given any finite set of points  $\{ E_j\}_{j=1}^N$ in   $(-2,2)$ with $0\notin \{ E_j\}_{j=1}^N+\{ E_j\}_{j=1}^N$,
we construct  potential  $V(n)=\frac{O(1)}{1+|n|}$ such that
$H=H_0+V$ has eigenvalues $\{ E_j\}_{j=1}^N$.
\item[5]Given any countable set of  points  $\{ E_j\}$ in  $(-2,2)$ with $0\notin \{ E_j\}+\{ E_j\}$,
and any function $h(n)>0$ going  to infinity arbitrarily slowly,
we construct  potential $|V(n)|\leq \frac{h(n)}{1+|n|}$ such that
$H=H_0+V$ has eigenvalues $\{ E_j\}$.
 \end{description}

%
%
%

\end{abstract}
\maketitle
\section{Introduction}

We consider the  discrete Schr\"odinger equation,
\begin{equation}\label{Gdis}
  (Hu)(n)=  u({n+1})+u({n-1})+V(n)u(n)=Eu(n),
\end{equation}
 where $V(n)$ is the perturbation.
 Denote by $H_0$   the free Schr\"odinger operator, namely,
 \begin{equation*}
    (H_0u)(n)= u({n+1})+u({n-1}).
 \end{equation*}
  Without loss of generality, we only consider the discrete Schr\"odinger operator in the half line $\N$.
  \begin{equation}\label{Gdis1}
   (Hu)(n)=  u({n+1})+u({n-1})+V(n)u(n)=Eu(n)\;\;\;\;\;(n\geq 1)
  \end{equation}
  with boundary condition
  \begin{equation}\label{Gbc}
    \frac{u(1)}{u(0)}=\tan\theta.
  \end{equation}
Similarly, we also have the  continuous Schr\"odinger operator $H=H_0+V(x)$.

Our goal is to identify the asymptotical behavior of the  potentials $V$ such that  there
is one   eigenvalue (finitely many eigenvalues, or infinitely many eigenvalues)   embedded into the absolutely continuous   and  essential spectra.
The identification of    eigenvalues/singular continuous spectrum embedded into ac (ess) spectrum    attracted much  attention   from different viewpoints, for example
\cite{sim07,kissc,kru,lotoreichik2014spectral,Naboko2018,lukwn,luk14,lukd1,simonov2016zeroes,KRS,KLS,MR2945209,remlingsharp}.

Suppose
   $ \limsup_{x\to \infty} x|V(x)|=a$. 
Our first interest is the  study of the sharp transition for the single eigenvalue embedded into the ac spectrum or ess spectrum.

Let us introduce the history of the continuous case first.
By a result of Kato \cite{kato}, there is no eigenvalue $E$ with $E>a^2$, which holds in any dimension.
From the classical Wigner-von Neumann  type functions
\begin{equation*}
    V(x)=\frac{c}{1+x}\sin(  kx+\phi),
\end{equation*}
we know that one can not do better than $\frac{a^2}{4}$.
For the
one dimensional case, Atkinson and Everitt  \cite{atk} obtained the optimal bound $\frac{4a^2}{\pi ^2}$, that is
there is no eigenvalue if  $E>\frac{4a^2}{\pi ^2}$ and   there are examples
with eigenvalues arbitrarily close to this bound.
The proof is based on Pr\"ufer transformation and sign  type potentials, that is
\begin{equation*}
    V(x)=\frac{c}{1+x}\text{ sgn } (  \sin (kx+\phi)).
\end{equation*}
 We refer the readers to Simon's paper for the full  history \cite{simon2017tosio}. Recently, we  constructed   examples such that  the optimal bound  $\frac{4a^2}{\pi ^2}$ can be achieved \cite{Lstark}.

One purpose of this paper is to obtain a similar sharp result for the discrete case. In   the following,   all the potentials  satisfy
 \begin{equation*}
   | V(n)|=\frac{O(1)}{1+n^{\frac{3}{4}}}.
 \end{equation*}
  By Weyl's theorem,  the essential spectrum of $H_0+V$ equals  $[-2,2]$. Moreover, the interval $(-2,2)$ is covered with absolutely continuous spectrum, cf. \cite{DK}.
 For any $E\in(-2,2)$, let $E=2\cos \pi k(E)$ with $k\in(0,1)$.  Sometimes, we use $k$ for simplicity.

 {\bf Definition.}
 We say $E\in(-2,2)$ is of rational (resp. irrational) type if $k(E)$ is rational (resp. irrational).
  We say  $ E\in(-2,2)$ is of rational type with even (odd) denominator if the denominator of   rational number $k(E)$ is even (odd).

However, the question of sharp bounds for embedded  eigenvalues  in discrete case  is much more delicate than   in the continuous case. The bounds heavily depend on the arithmetic property  of $k(E)$.
If $k(E)$ is irrational, Remling's arguments  imply $H=H_0+V$ does not have eigenvalue in $(-2\sqrt{1-A^2a^2},2\sqrt{1-A^2a^2})$ with $A=\frac{2}{\pi}$ \cite{remlingsharp}.
Like the continuous case,   Wigner-von Neumann  type functions
\begin{equation*}
    V(n)=\frac{c}{1+n}\sin(  kn+\phi),
\end{equation*}
can only give  the   bound $A=\frac{1}{2}$ (see Lemma \ref{Keylemma}). Thus there is a gap between $\frac{1}{2}$ and $\frac{2}{\pi}$.
We use the Pr\"ufer transformation (cf. \cite{remlingsharp,KLS,KRS}) and sign  type potentials for the discrete case to show that $\frac{2}{\pi}$ is sharp. See Theorems \ref{thm1}-\ref{thm3}.

The  most important contribution of the present paper is to study the sharp bounds for rational $k(E)$, which is missing in     previous literature.
Suppose the denominator of $k(E) $ is $q$. The  average  of $|\sin(2\pi k(E)n+\phi)|$ with respect to $n$ over  each period $q$  depends on the initial phase $\phi$, which is different from the irrational $k(E)$.
Predicting  the  sharp bounds is  the first challenge since there is no ergodic theorem at hand.
More importantly,
it is very easy to break the initial phase in   each period.  Another issue   is that there are two transition lines for the discrete case (see Theorems \ref{thm2} and \ref{thm3}). For the continuous Schr\"odinger operator, the average does not depend on energies $E$,   so  it is   easier to deal with.
 Thus the problem of embedded eigenvalues for the  discrete case has significant new challenges.

We distinguish the eigenvalues by  the denominator of $k(E)$ is even and   the denominator of $k(E)$ is odd.
Obtaining the sharp transitions for the rational type eigenvalues depends on whether we can construct potentials in each period and control the initial phase $\phi$ after each
period $q$.

For even denominator case,  our idea is to choose a good potential $V$ ( half of the  values of    $V(n)$ are positive and the remaining half are negative in each period)
to create some cancellation so that  we can almost keep the initial phase $\phi$ after each period.
It is very difficult to construct a potential which creates  the cancellation,  does not change the initial $\phi$  and   achieves  the optimal bounds at the same time.
    We address this problem by taking the second leading entry of the equation of  the Pr\"ufer angle (see \eqref{PrufT}) into consideration, which is more delicate than the usual method.
Even when using this way, there are    two issues to be addressed.  The first issue is that   $\cot\pi x$  is a singular function so it is difficult  to control the derivative. Luckily, the trajectories of    $\{\phi+jk\}_{j\in\N}$ that   we must choose (in order to achieve the optimal bound) can be shown  to be far way from the singular points of functions $\cot\pi x$ (see \eqref{G4equ5dec8ap34}).
The second issue is that  the usual  sign  type potentials cannot change the initial phase much after one period. However, it will destroy the initial phase after plenty of  periods  since there is no    full cancellation.
We adapt the potentials a little bit (we call the result {\it almost sign  type potentials})
to create  the full cancellation  by solving an  algebraic equation (see \eqref{Gapr6algebraic}).

For the odd case,  it is impossible to create such  cancellation.
Thus the initial phases will change in every period.
By  some delicate estimates,   we get two nice bounds  ${B}_q $  and $A_q$, where $q$ is the denominator of $k(E)$.
See \eqref{GAqeven}, \eqref{GAqodd}  and \eqref{GAqoddprime} for the definitions of $B_q $  and $A_q$.
However, there is still a gap between $B_q $  and $A_q$.
We should mention that $B_q $  and $A_q$ are close. In particular, they share the same asymptotic-$\frac{2}{\pi}$ as $q$ goes to infinity, which is exactly the bound for irrational case.
Also, $B_q >\frac{1}{2}$ (see \eqref{GAqodddec9prime}).  It  means that the bound we get is  better than that given by Wigner-von Neumann  type functions.

 Another  interest of ours  is to investigate the  distribution  of  eigenvalues embedded into $(-2,2)$. Under the assumption that $\limsup_{n\to\infty} n|V(n)|<\infty$,
 the possible limits of the embedded eigenvalues for all the boundary conditions are $-2$ and $2$.
 We obtained the asymptotical behaviors of $|E_i\pm 2|$. See Theorem \ref{thm6}.

  Our last result  is to construct  finitely or infinitely many eigenvalues embedded into $(-2,2)$.
 For the continuous case, Simon \cite{simdense} and Naboko \cite{nabdense} constructed  potentials such that dense eigenvalues can be embedded in  absolutely continuous spectrum (or essential  spectrum).
 For the discrete case, Naboko and Yakovlev \cite{naboko1992point} constructed potentials $V$ such that $H_0+V$ has the given eigenvalues.
 But the  rational independence of $k(E)$ was needed  in their   construction. Remling \cite{Remdis} constructed  power decaying potentials $V$ such that  \eqref{Gdis1} has an $\ell^2(\N) $ for a full Lebesgue measure set of $E\in(-2,2)$ (or singular continuous spectrum).
 Recently, Jitomirskaya and Liu \cite{jl} introduced   piecewise functions  to construct potentials such that $H_0+V$ has the given eigenvalues without any rationally  independent assumption, which works for manifolds \cite{jl} and perturbed periodic operators \cite{ld}.
 We develop  similar ideas to deal with the discrete Schr\"odinger operator.
 See  Theorems \ref{thm4} and \ref{thm5}.
In   our other two papers,  we    will  use piecewise functions  to construct perturbed periodic Jacobi operators with embedded eigenvalues \cite{LDJ} and  perturbed Stark type operators with embedded eigenvalues \cite{Lstark}.  Although some ideas of corresponding  construction in this paper are from  \cite{simdense,jl,LDJ,ld}, there are  several new ingredients.
 Comparing to \cite{LDJ} and \cite{ld}, the potentials here  are given  in an explicit way   (piecewise  Wigner-von Neumann  type functions).
 Since   in the discrete case,   potentials with support in any interval $[a,b]$    is a space of  finite dimension,  it is difficult to apply  Simon's construction  to the discrete case \cite{simdense}.  See Remarks 2.15 and 2.16 for more details.


 \section{Main results}\label{Smain}
 For any $E\in(-2,2)$, let $E=2\cos \pi k$ with $k\in(0,1)$.  Sometimes, we use $k(E)$ to indicate the dependence.
 In the rest of this paper, $E$ is always in $(-2,2)$ and $k$ is in $(0,1)$.

 Define $S_q\subset (-2,2)$ for $q\in\N \backslash \{1\}$,
 \begin{equation}\label{Gsq}
    S_q=\{E\in(-2,2): k(E)=\frac{p}{q},  \text{ for some } p=1,2,\cdots,q-1 \text { and } \text{ gcd } (p,q)=1\}.
 \end{equation}

 For even   $q\geq 2$,
 let
 \begin{equation}\label{GAqeven}
    A_q= \frac{2}{q\sin\frac{\pi}{q}}.
 \end{equation}
  For odd   $q\geq 3$,
 let
 \begin{equation}\label{GAqodd}
    A_q= \frac{2\cos\frac{\pi}{2q}}{q\sin\frac{\pi}{q}},
 \end{equation}
 and
 \begin{equation}\label{GAqoddprime}
    B_q = \frac{1+\cos\frac{\pi}{q}}{q\sin\frac{\pi}{q}}.
 \end{equation}
 Denote by
 \begin{equation}\label{Gs0}
    S_0=\{E\in(-2,2): k(E) \text{ is irrational }\},
 \end{equation}
and
 \begin{equation}\label{GA0}
    A_0= \frac{2}{\pi}.
 \end{equation}
We should mention that there are  no definitions for  $S_1$ and $A_1$.  In the following, we always assume $q\neq 1$.

{\bf Remark:}
\begin{itemize}
  \item By the definition of $A_q$ and $B_q $, one has
\begin{equation}\label{Gasy}
  \lim_{q\to\infty}  B_q =  \lim_{q\to\infty}  A_q=\frac{2}{\pi}.
\end{equation}

  \item We also have for odd $q\geq 3$ (see \eqref{GAqodddec9prime}),
\begin{equation}\label{Gbetter}
     B_q >\frac{1}{2}.
\end{equation}
\end{itemize}
 \begin{theorem}\label{thm1}
Suppose potential $V$  satisfies
\begin{equation*}
   \limsup_{n\to \infty}|nV(n)|=a<\frac{1}{A_q}.
 \end{equation*}
 Let $ {E}_q=2\sqrt{1-a^2A_q^2}$.
  Then for any boundary condition \eqref{Gbc}, the operator $H =H_0+V$ given by \eqref{Gdis1} does not admit any  eigenvalue  in $(-E_q,E_q)\cap S_q$.
 \end{theorem}
\begin{remark}
Remling's argument implies Theorem  \ref{thm1} for the case $q=0$  \cite{remlingsharp}. We list the result and also give the  proof in this paper for completeness.
\end{remark}
 \begin{theorem}\label{thm2}
Suppose $q$ is even
and   $a>\frac{1}{A_q}$.
Then for any $\theta\in[0,\pi]$ and   $ E\in S_q$,  there exist      potentials  $V$
such that $\limsup_{n\to \infty}|nV(n)|=a $  and $E$ is an eigenvalue of the  associated  Schr\"odinger operator   $H=H_0+V$ with boundary condition \eqref{Gbc}.
 \end{theorem}
 \begin{theorem}\label{thm3}
Suppose $q$ is even.
 For  any  $0<a<\frac{1}{A_q}$, let $ {E}_q=2\sqrt{1-a^2A_q^2}$.
 Suppose  $E\in S_q$  and  $E\in(-2,-E_q)\cup (E_q,2)$. Then for any $\theta\in[0,\pi]$, there exist      potentials  $V$
such that $\limsup_{n\to \infty}|nV(n)|=a $  and $E$ is an eigenvalue of the  associated  Schr\"odinger  operator   $H=H_0+V$ with boundary condition \eqref{Gbc}.
 \end{theorem}

 From Theorems \ref{thm1}, \ref{thm2} and \ref{thm3}, we see that for even $q$, $\frac{1}{A_q}$ and  $E_q$ are the sharp transitions   of eigenvalues embedded into ac spectrum for $E\in S_q$.
 Moreover, we have the following two  interesting  corollaries,
 \begin{corollary}\label{cor1}
 Suppose potential $V$  satisfies
\begin{equation*}
   \limsup_{n \to \infty}|n V(n)|=a<1.
 \end{equation*}
 Then  for any boundary condition \eqref{Gbc},  $H =H_0 +V$  does not have any  eigenvalue in $(-2\sqrt{1-a^2},2\sqrt{1-a^2})$.
 \end{corollary}
 \begin{corollary}\label{cor2}
Suppose  $a>1$.  Then for any boundary condition  \eqref{Gbc}, there exist  potentials $V$ such that
\begin{equation*}
   \limsup_{n \to \infty}|n V(n)|=a,
 \end{equation*}
and
    $H =H_0+V$    has    some  eigenvalue $E\in(-2,2)$.
 \end{corollary}
 We now turn to the case of odd $q$.



 \begin{theorem}\label{thm2odd}
Suppose $q$ is odd
and   $a>\frac{1}{ B_q }$.
Then for any boundary condition \eqref{Gbc}  and any $ E\in S_q$,  there exist      potentials  $V$
such that $\limsup_{n\to \infty}|nV(n)|=a $  and the  associated  Schr\"odinger operator    $H=H_0+V $  has  eigenvalue $E$.
 \end{theorem}
 \begin{theorem}\label{thm3odd}
Suppose $q$ is odd.
 For any    $0<a<\frac{1}{B_q }$, let $ \tilde{E}_q=2\sqrt{1-a^2B_q ^2}$.
 Suppose  $E\in S_q$  and  $E\in(-2,-\tilde E_q)\cup (\tilde E_q,2)$. Then  for any $\theta\in[0,\pi]$, there exist      potentials  $V$
such that $\limsup_{n\to \infty}|nV(n)|=a $  and the  associated  Schr\"odinger operator    $H=H_0+V $  with boundary condition \eqref{Gbc} has  eigenvalue $E$.
 \end{theorem}
By Theorems \ref{thm1}, \ref{thm2odd} and \ref{thm3odd}, there is a gap between $\frac{1}{A_q}$ and $\frac{1}{B_q}$ for odd $q$.

Denote by
\begin{equation*}
  P=\{E\in(-2,2): \eqref{Gdis1} \text { has an } \ell^2(\N) \text{ solution }\}.
\end{equation*}

$P$ is the collections of    the eigenvalues of $H_0+V$ with all the possible boundary conditions at $0$.

\begin{corollary}\label{cordec}
Suppose potential $V$  satisfies
\begin{equation*}
   \limsup_{n\to \infty}|nV(n)|=a<\frac{\pi}{2}.
 \end{equation*}
  Then for any $\epsilon>0$, 
  $P\cap (-2\sqrt{1-\frac{4a^2}{\pi^2}}+\epsilon,2\sqrt{1-\frac{4a^2}{\pi^2}}-\epsilon)$ is a finite set.
\end{corollary}
\begin{remark}
Under the assumption of Corollary \ref{cordec}, Remling \cite{remlingsharp} showed  that  $H_0+V$  has
no singular continuous spectrum in  $(-2\sqrt{1-\frac{4a^2}{\pi^2}},2\sqrt{1-\frac{4a^2}{\pi^2}})$.
\end{remark}
%
The  asymptotical behaviors of eigenvalues lying outside $[-2,2]$ has  been well studied since it is related to a lot of topics in spectral theory, for example  the
regular behavior of the spectral measure near the points $-2$ and $2$ and the problems of purely ac spectrum on $[-2,2]$ \cite{dks,sa,dk04,dc07,damanik2003variational,killip2003sum,deift1999absolutely}.
Our next result is to investigate the asymptotical behaviors of eigenvalues close to the boundaries  $-2$ and $2$ lying in $(-2,2)$.

\begin{theorem}\label{thm6}
Suppose
\begin{equation*}
    \limsup_{n\to \infty} n|V(n)|=a.
\end{equation*}
Then $P$ is a countable set with two possible accumulation points $2$ and $-2$.  Moreover, the following estimate holds,
\begin{equation*}
    \sum_{E_i\in P}(4-E_i^2)\leq 4a^2+4\min\{1,a\}.
\end{equation*}

\end{theorem}
 Theorem  \ref{thm6} implies the speed of    $E_i\in P$ going to the   boundaries $\pm 2$ behaves
        $|E_i-2|\approx\frac{1}{1+ {i}}$ ($|E_i+2|\approx\frac{1}{1+{i}}$).
        Remling \cite{remling2000schrodinger} showed that in the continuous case, it is impossible to  improve it to $|E_i\pm 2|\approx\frac{1}{1+ {i}^{1+\epsilon}}$. This means that the bound in Theorem  \ref{thm6} is optimal in some sense.

The proof of Theorem \ref{thm6}  is motivated by \cite{KRS}.
The key idea of  \cite{KRS} is to show the almost orthogonality of $\theta(n,k(E_1))$ and $\theta(n,k(E_2))$, where $\theta(n,k(E_1))$ ($\theta(n,k(E_2))$) is the  Pr\"ufer angle with respect to energy $E_1$ ($E_2$).
  For the discrete case, it is hard to verify the  almost orthogonality. Luckily, a weaker version of  almost orthogonality in the discrete setting has been obtained  in \cite{LDJ}, which is enough to hand our problem here.

Our next two  results are to construct potentials with finitely  many (countable) eigenvalues embedded into $(-2,2)$.
For a  set $A\subset \R$, denote by $A+A=\{x+y: x\in A \text { and } y\in A\}$.
  \begin{theorem}\label{thm4}

Given any finite set of points  $A=\{ E_j\}_{j=1}^N$ in   $(-2,2)$ with $0\notin A+A$ and $\{\theta_j\}_{j=1}^N\subset[0,\pi]$,  there exist  potentials  $V(n)=\frac{O(1)}{1+n}$ such that
for each $E_j\in A$,
\eqref{Gdis} has an $\ell^2(\N)$ solution with boundary condition $\frac{u(1)}{u(0)}=\tan\theta_j$.

 \end{theorem}
 \begin{theorem}\label{thm5}
 Given any countable set of  points  $B=\{ E_j\}$ in  $(-2,2)$ with  $0\notin B+B$, any sequence $\{\theta_j\}_j\subset [0,\pi]$ and
  any function $h(n)>0$ going  to infinity arbitrarily slowly,
there exist    potentials  $|V(n)|\leq \frac{h(n)}{1+n}$ such that
for each $E_j\in A$,
\eqref{Gdis} has an $\ell^2(\N)$ solution with boundary condition $\frac{u(1)}{u(0)}=\tan\theta_j$.
 \end{theorem}
 \begin{remark}
 In \cite{LDJ}, Theorems \ref{thm4} and \ref{thm5} have been proved for perturbed periodic Jacobi operators. However, the explicit formula for the potentials can not be given. We will use the piecewise  Wigner-von Neumann  type functions to complete our construction in this paper, which we believe to be of independent interest.
 \end{remark}
 \begin{remark}
 For the continuous case, Simon \cite{simdense} used  Wigner-von Neumann  type functions
 $V(x)=\frac{a}{1+x}\sum_{j} \sin(2\lambda_jx +2\phi_j)\chi_{[a_j,\infty)}$,
and functions $W$ with support in $(1,2)$ to do the construction.
For the continuous case, we can adapt potential $W$ to match the boundary condition $\theta_j$.
In the discrete case, this is impossible. 
 \end{remark}
 The rest of the  paper is organized in the following way.
In Section \ref{Absence}, we will show the absence of embedded eigenvalues   if our perturbation is small, and finish the proof of  Theorem \ref{thm1} and Corollary \ref{cor1}.
In Section \ref{even}, we give some preparations for the proof of the rational type eigenvalues  with even denominators.
In   Section \ref{embedone}, we will construct potentials  such that the associated operators have   one embedded eigenvalue, and prove
 Theorems \ref{thm2}, \ref{thm3}, \ref{thm2odd}, \ref{thm3odd}, and Corollaries \ref{cor2} and \ref{cordec}.
 In Section \ref{SecAsym}, we will prove Theorem \ref{thm6}.
 In Section \ref{Smany},
 we will construct potentials such that the associated operators have  finitely  (countably) many embedded eigenvalues.
\section{Proof of Theorem \ref{thm1} and Corollary \ref{cor1}}\label{Absence}
Let us introduce the Pr\"ufer transformation first (cf. \cite{remlingsharp,KLS,KRS}).
Suppose $u(n,E)$ is a solution of \eqref{Gdis1} with $u(0,E)=0$ and $ u(1,E)=1$.
We do not make the difference between   $u(n,k(E))$, $u(n,k)$  and $u(n,E)$.
In   Sections \ref{Absence}, \ref{even} and \ref{embedone}, all the  potential $V$  satisfy
\begin{equation}\label{Gbdp}
    |V(n)|=\frac{O(1)}{1+n}.
\end{equation}

Let
\begin{equation}\label{L2}
    Y(n,k)=\frac{1}{\sin \pi k} \left(
                                  \begin{array}{cc}
                                    \sin \pi k & 0 \\
                                    -\cos \pi k & 1 \\
                                  \end{array}
                                \right)\left(\begin{array}{c}
                                         u(n-1,k) \\
                                         u(n,k)
                                       \end{array}\right).
\end{equation}

Define the Pr\"ufer variables $R(n,k)$ and $\theta(n,k)$ as
\begin{equation}\label{L21}
    Y(n,k)=R(n,k)\left(\begin{array}{c}
                                        \sin(\pi \theta(n,k)-\pi k) \\
                                        \cos(\pi \theta(n,k)-\pi k)
                                       \end{array}\right).
\end{equation}
It is well known that $R$ and $\theta$ obey the equations
\begin{equation}\label{PrufR}
    \frac{R(n+1,k)^2}{R(n,k)^2}=1-\frac{V(n)}{\sin \pi k}\sin 2\pi \theta(n,k)+\frac{V(n)^2}{\sin^2\pi k}\sin^2\pi \theta(n,k)
\end{equation}
and
\begin{equation}\label{PrufT}
    \cot (\pi \theta(n+1,k)-\pi k)=\cot \pi \theta(n,k)-\frac{V(n)}{\sin \pi k}.
\end{equation}
We will give some Lemmas, which are useful in the proof of main Theorems.
\begin{lemma}\cite[Prop.2.4]{KLS}\label{Leap3}
Suppose $\theta(n,k)$ satisfies \eqref{PrufT} and $|\frac{V(n)}{\sin \pi k}|<\frac{1}{2}$. Then we have
\begin{equation}\label{Gap3}
 | \theta(n+1,k)- k-\theta(n,k)|\leq \left|\frac{V(n)}{\sin \pi k}\right|.
\end{equation}
\end{lemma}
The following Lemma is a improvement of Lemma \ref{Leap3}   if $\theta(n,k)$ is far way from the singular points of $\cot \pi x$.
\begin{lemma}\label{Leapr31}
Suppose
\begin{equation}\label{Gap312}
  \frac{1}{|\sin\pi \theta(n,k)|}=O(1).
\end{equation}
Then under the condition of \eqref{Gbdp}, we have for large $n$,
  \begin{equation*}
       \theta(n+1,k)= k+  \theta(n,k)+ \sin^2\pi  \theta(n,k)\frac{V(n)}{ \pi\sin \pi k}+\frac{O(1)}{1+n^2}.
\end{equation*}
\end{lemma}
\begin{proof}
Let $\theta_0=\theta(n,k)$ and $\theta_1=\theta(n+1,k)$.
By Lemma \ref{Leap3}, one   has
\begin{equation}\label{Gap311}
  \theta_1= k+  \theta_0+\frac{O(1)}{1+n}.
\end{equation}
Let $f(x)=\cot\pi x$.
By the assumption \eqref{Gap312} and \eqref{Gap311}, one has for large $n$,
\begin{equation}\label{Gap313}
  f^{\prime\prime}(x)=O(1).
\end{equation}
for all $x \in[\theta_0,\theta_1-k]$ or $x \in[\theta_1-k,\theta_0]$.

 By \eqref{PrufT}, one has
\begin{equation}\label{Gap315}
f(\theta_1-k)=f(\theta_0)-\frac{V(n)}{\sin \pi k}.
\end{equation}
Using the Tayor series and \eqref{Gap313}, one has
\begin{equation}\label{Gap314}
 f(\theta_1-k)=f(\theta_0)+f^{\prime}(\theta_0)(\theta_1-k-\theta_0)+O(1)(\theta_1-k-\theta_0)^2.
\end{equation}
Notice that
\begin{equation}\label{Gde}
  f^\prime(x)=-\frac{\pi}{\sin^2\pi x}.
\end{equation}
The Lemma follows \eqref{Gap311}, \eqref{Gap315}, \eqref{Gap314} and \eqref{Gde}.
\end{proof}
\begin{lemma}\label{Lemax}
 Let
\begin{equation}\label{GAq}
  \tilde{A}_q= \frac{1}{q}\max_{\phi\in[0,2\pi]}\sum_{j=0}^{q-1}|\sin(\frac{2\pi}{q}j+\phi)|.
\end{equation}
Then  for all   $q\geq 2$,
\begin{equation}\label{Gequal}
     \tilde{A}_q=A_q.
\end{equation}
Moreover, for even $q$,
\begin{equation}\label{GAqevendec9}
    A_q= \frac{2}{q}\sum_{j=0}^{\frac{q}{2}-1}\sin(\frac{2\pi}{q}j+\frac{\pi}{q}).
 \end{equation}
 For odd $q$, we have
 \begin{equation}\label{GAqodddec9prime}
    B_q=  \frac{1}{q}\min_{\phi\in[0,2\pi]}\sum_{j=0}^{q-1}|\sin(\frac{2\pi}{q}j+\phi)|>\frac{1}{2}.
 \end{equation}
\end{lemma}
\begin{proof}
It is well known that
\begin{equation}\label{Gsum}
    \sin(a)+\sin(a+x)+\cdots+\sin(a+(n-1)x)=\frac{\cos(a-\frac{x}{2})-\cos(a+(n-\frac{1}{2})x)}{2\sin\frac{x}{2}}.
\end{equation}

Let us consider the even case first.
By the definition of $\tilde{A}_q$, one has for even $q$,
\begin{equation}\label{Gequdec91}
      \tilde{A}_q= \frac{2}{q}\max_{\phi\in[0,\frac{2\pi}{q}]}\sum_{j=0}^{\frac{q}{2}-1}\sin(\frac{2\pi}{q}j+\phi),
\end{equation}
and for odd $q$,
\begin{equation}\label{Gequdec91odd}
      \tilde{A}_q= \frac{1}{q}\max_{\phi\in[0,\frac{\pi}{q}]} \left(\sum_{j=0}^{\frac{q-1}{2}}\sin(\frac{2\pi}{q}j+\phi)- \sum_{j=\frac{q+1}{2}}^{q-1}\sin(\frac{2\pi}{q}j+\phi)\right).
\end{equation}
Applying \eqref{Gsum}, one has
\begin{equation}\label{Gequdec92}
    \sum_{j=0}^{\frac{q}{2}-1}\sin(\frac{2\pi}{q}j+\phi)=\frac{\cos(\phi-\frac{\pi}{q})}{\sin\frac{\pi}{q}}.
\end{equation}
\eqref{Gequal} and \eqref{GAqevendec9} follows from \eqref{Gequdec91} and \eqref{Gequdec92}.

Now let us consider the odd $q$.
Applying \eqref{Gsum}, one has
\begin{equation}\label{Gequdec94}
  \sum_{j=0}^{\frac{q-1}{2}}\sin(\frac{2\pi}{q}j+\phi)- \sum_{j=\frac{q+1}{2}}^{q-1}\sin(\frac{2\pi}{q}j+\phi)=
   \frac{\cos\phi+\cos(\phi-\frac{\pi}{q})}{\sin\frac{\pi}{q}}.
\end{equation}
\eqref{Gequdec94} achieves the maximum at $\phi=\frac{\pi}{2q}$  and the minimum at $\phi=0$. It leads to
\eqref{Gequal} and the  equality part of \eqref{GAqodddec9prime}.

We will prove the inequality part of  \eqref{GAqodddec9prime}.
It immediately follows from
\begin{eqnarray*}
   B_q   &=& \frac{1}{q} \sum_{j=0}^{q-1}|\sin\frac{2\pi}{q}j| \\
   &>&  \frac{1}{q} \sum_{j=0}^{q-1}\sin^2\frac{2\pi}{q}j=\frac{1}{2}.
\end{eqnarray*}

\end{proof}
\begin{lemma}\label{Leabsence1}
Suppose $k$ is irrational. Then for any $\varepsilon>0$, there exists some $N>0$ such that  for large $n_0$,
\begin{equation}\label{Gabsence1}
 \left|\left(\frac{1}{N} \sum_{n=n_0}^{n=n_0+N-1} |\sin 2\pi \theta(n,k)|\right)-\frac{2}{\pi}\right|\leq \varepsilon.
\end{equation}

\end{lemma}
\begin{proof}
Notice that
\begin{equation*}
  \int_0^1|\sin2\pi \varphi|d\varphi=\frac{2}{\pi}.
\end{equation*}
By the Ergodicity of irrational rotation $ k$, for any $\varepsilon>0$, there exists some $N>0$ such that
for any $n_0$ and $\phi\in[0,2\pi)$, we have
\begin{equation}\label{Geo1}
     \left|\left(\frac{1}{N} \sum_{n=n_0}^{n=n_0+N-1} |\sin (2\pi k n+\phi)|\right)-\frac{2}{\pi}\right|\leq \varepsilon.
\end{equation}
By \eqref{Gbdp}, \eqref{PrufT} and \eqref{Gap3},  one has for large $n_0$,
\begin{equation}\label{Geo2}
    |  \theta(n^\prime,k)- (n^\prime -n_0)k -\theta(n_0,k)|\leq \varepsilon
\end{equation}
for all $n_0\leq n^\prime\leq n_0+N$.
Now the Lemma follows from \eqref{Geo1} and \eqref{Geo2}.

\end{proof}
\begin{lemma}\label{Leabsence2}
Suppose $k\in S_q$ with $q\geq 2$. Then for any $\varepsilon>0$ and   large $n_0$,
\begin{equation}\label{Gabsence2}
  \frac{1}{q} \sum_{n=n_0}^{n=n_0+q-1} |\sin 2\pi \theta(n,k)| \leq A_q+ \varepsilon,
\end{equation}
and for odd $q\geq 3$,
\begin{equation}\label{Gabsence2dec10}
  \frac{1}{q} \sum_{n=n_0}^{n=n_0+q-1} |\sin 2\pi \theta(n,k)| \geq B_q- \varepsilon.
\end{equation}
\end{lemma}
\begin{proof}
By  the definition of $A_q$ and \eqref{Gequal},  for any  $n_0$ and $\phi\in[0,2\pi)$, we have
\begin{equation}\label{Gra1}
     \frac{1}{q} \sum_{n=n_0}^{n=n_0+q-1} |\sin (2\pi  kn+\phi)|\leq A_q.
\end{equation}

By \eqref{Gbdp}, \eqref{PrufT}  and \eqref{Gap3} again,  one has for large $n_0$,
\begin{equation}\label{Gra2}
    |  \theta(n^\prime,k)- (n^\prime -n_0)k -\theta(n_0,k)|\leq \varepsilon
\end{equation}
for all $n_0\leq n^\prime\leq n_0+q$.
Now \eqref{Gabsence2} follows from \eqref{Gra1} and \eqref{Gra2}.
Similarly, \eqref{Gabsence2dec10} follows from
\eqref{GAqodddec9prime}.
\end{proof}

\begin{proof}[\bf Proof of Theorem  \ref{thm1}]
By the assumption of Theorem \ref{thm1}, we have for any $\varepsilon>0$,
\begin{equation}\label{Gbdv}
    |V(n)|\leq \frac{a+\varepsilon}{1+n} \text { for large } n.
\end{equation}

We first consider $E\in S_0$, i.e., $k(E)$ is irrational.
By \eqref{PrufR} and \eqref{Gbdp}, one has
\begin{equation}\label{PrufR1}
  \ln R(n+1,k)^2 -\ln R(n,k)^2=-\frac{V(n)}{\sin \pi k}\sin 2\pi \theta(n,k)+\frac{O(1)}{n^2+1}.
\end{equation}

Denote by   $\lfloor x\rfloor$ be the largest integer less or equal than $x$. Assume $n_0$ is large enough.
By \eqref{Gabsence1},
we have for all $n>n_0$,
\begin{eqnarray}
    \ln R(n+1,k)^2  &\geq & \ln R(n_0,k)^2-\sum_{j=1}^{n} \frac{(a+\varepsilon)}{(1+j)\sin \pi k}|\sin 2\pi \theta(j,k)|-\sum_{j=n_0}^{n} \frac{O(1)}{j^2+1} \nonumber\\
   &\geq & -C(k,n_0,a) -\frac{(a+\varepsilon)}{\sin \pi k}\sum_{i=1}^{\lfloor\frac{n-n_0}{N}\rfloor}\sum_{m=n_0+(i-1)N}^ {m=n_0+i N-1}\frac{|\sin 2\pi \theta(m,k)|}{m+1}\nonumber\\
&\geq & -C(k,n_0,a)-\frac{(a+\varepsilon)}{\sin \pi k}\sum_{i=1}^{\lfloor\frac{n-n_0}{N}\rfloor}\frac{1}{n_0+(i-1)N}\sum_{m=n_0+(i-1)N}^ {m=n_0+i N-1}|\sin 2\pi \theta(m,k)| \nonumber\\
 &\geq & -C(k,n_0,a)-\frac{(a+\varepsilon)}{\sin \pi k} (\frac{2}{\pi}+\varepsilon) \sum_{i=1}^{\lfloor\frac{n-n_0}{N}\rfloor}\frac{N}{n_0+(i-1)N}\nonumber\\
&\geq & -C(k,n_0,a) -\frac{(a+\varepsilon)}{\sin \pi k} (\frac{2}{\pi}+\varepsilon) \ln n.\label{Gequ1}
\end{eqnarray}
Since $E=2\cos \pi k$ and $|E|<E_0$ with $E_0=2\sqrt{1-a^2A_0^2}$, we have for small enough $\varepsilon>0$,
\begin{equation}\label{Gequ2}
    \frac{(a+\varepsilon)}{\sin \pi k} (\frac{2}{\pi}+\varepsilon)<1.
\end{equation}
Thus by \eqref{Gequ1} and \eqref{Gequ2},
we have for large $n$,
\begin{equation*}
  R^2(n,k)\geq \frac{1}{C n}.
\end{equation*}
This implies
 $R(n,k)$ is not in $\ell^2(\N)$.
By \eqref{L2} and \eqref{L21}, we have that
$u(n,k)$ is not in  $\ell^2(\N)$. We finish the proof for the irrational $k $.

Assume $E\in S_q$ for $q\geq 2$. The proof is similar. We only need to replace $ N$ with $q$ and \eqref{Gabsence1}  with  \eqref{Gabsence2}.

\end{proof}
\begin{proof}[\bf Proof of Corollary \ref{cor1}]
By the definition of $A_q$, one has
\begin{equation*}
    A_q\leq 1 \text { for all possible } q.
\end{equation*}
Now Corollary \ref{cor1} follows from Theorem \ref{thm1}.
\end{proof}
\section{ Preparations for the rational type eigenvalues  with even denominators}\label{even}
In this section,  we consider $k\in S_q$ with  even $q\geq 2$.

By the definition of $A_q$ ($q\geq 2$) and Lemma \eqref{Lemax},  for any $\varepsilon>0$, there exists $\delta$ such that
\begin{equation}\label{G4equ1}
    \frac{1}{q}\sum_{j=0}^{q-1}|\sin(\frac{2\pi}{q}j+\phi)|\geq A_q-\varepsilon
\end{equation}
holds for all $\phi\in(\frac{\pi}{q}-2\pi\delta,\frac{\pi}{q}+2\pi\delta)\subset (0,\frac{2\pi}{q})$, where $\delta$ is small enough (will be determined soon).

Suppose $k=\frac{p}{q}$ with coprime $p$ and $q$.
Let (by the fact $q$ is even)  $p_1^+ ,p_2^+,\cdots,p_{\frac{q}{2}}^+$ and $p_1^-,p_2^-,\cdots,p_{\frac{q}{2}}^-$ be a permutation of $0,1,2,\cdots,p-1$
 such that  for all $\phi\in(\frac{\pi}{q}-2\pi\delta,\frac{\pi}{q}+2\pi\delta)$,
\begin{equation}\label{G4equ2}
    \sin(\frac{2\pi p}{q}p_j^++\phi)>0
\end{equation}
and
\begin{equation}\label{G4equ3}
   \sin(\frac{2\pi p}{q}p_j^-+\phi)<0
\end{equation}
for $j=1,2,\cdots,\frac{q}{2}$.
Actually,
\begin{equation}\label{Gdec8n1}
    p_{j}^+\frac{p}{q}= \frac{j-1}{q} \mod\Z
\end{equation}
and
\begin{equation}\label{Gdec8n2}
    p_{j}^-\frac{p}{q}=\frac{1}{2}+ \frac{j-1}{q} \mod\Z
\end{equation}
for $j=1,2,\cdots,\frac{q}{2}$.

Now we are in the position to  construct $V$.
Let  $n_0$ be a large fixed positive integer.
Define $V(n)=0$ for all $n\leq n_0-2$. Let  $V(n_0-1)$ be such that
\begin{equation}\label{G4equ4}
    \theta(n_0,k)=\frac{1}{2q}.
\end{equation}
Suppose $a>0$.
We will  define  for $m\geq 0$,
\begin{equation}\label{G4equ6}
  V(n_0+qm+p_j^+)=\frac{a_{m,j}^+}{1+n_0+qm}
\end{equation}
and
\begin{equation}\label{G4equ7}
    V(n_0+qm+p_j^-)=-\frac{a_{m,j}^-}{1+n_0+qm},
\end{equation}
where $a_{m,j}^{\pm}>0$ is close to $a$. We will give the values of $a_{m,j}^{\pm}$ later.
\begin{theorem}\label{Ledec9}
Let $k\in S_q$ with even $q\geq2$.
Let  $m\geq0$. Suppose $V(n-1)$ and $\theta(n,k)$ are defined for all $n\leq n_0+mq$. Suppose $\theta(n_0+mq,k)\in(\frac{1}{2q}-\delta,\frac{1}{2q}+\delta)$.
Then there exists $a_{m,j}^{\pm}$, $j=1,2,\cdots,\frac{q}{2}$ and $\tilde{\delta}$ (small enough depending  on $\delta$ and $\tilde{\delta}\to 0$ as $\delta\to 0$ ) such that the following  statements hold,
\begin{description}
  \item[Nominators $a_{m,j}^{\pm}$]   for all $j=1,2,\cdots,\frac{q}{2}$,
  \begin{equation*}
  | a_{m,j}^{\pm}-a|\leq \tilde{\delta}.
  \end{equation*}

  \item[Potentials $V(n)$ on $(n_0+mq-1,n_0+mq+q)$]   Let
  \begin{equation*}
    V(n_0+mq+p_j^{\pm})= \pm \frac{a_{m,j}^{\pm}}{1+n_0+mq}
  \end{equation*}
  and define $\theta(n,k)$ by \eqref{PrufT} for $n_0+mq +1\leq n\leq n_0+mq+q$. Then
  \begin{equation}\label{Gequ20}
    \theta(n_0+mq+q,k)=qk+ \theta(n_0+mq,k)+\frac{O(1)}{1+n_0+m^2}.
  \end{equation}
\end{description}

\end{theorem}
\begin{proof}
  Recall that $k=\frac{p}{q}$ and the potential $V$ we constructed  will satisfy  \eqref{Gbdp}.

By \eqref{Gbdp} and \eqref{Gap3}
\begin{equation}\label{G4equ5dec8ap31}
       \theta(n+j,k)= j k+  \theta(n,k)+ \frac{O(1)}{1+n},
\end{equation}
for any $n$ and $1\leq j\leq q$.

By the definitions of $p^{\pm}_j$, one has
\begin{equation}\label{G4equ5dec8ap32}
       \theta(n_0+mq+p_j^{+},k)= \frac{j-1}{q}+  \theta(n_0+mq,k)+ \frac{O(1)}{1+n_0+m}
\mod \Z,\end{equation}
and
\begin{equation}\label{G4equ5dec8ap33}
       \theta(n_0+mq+p_j^{-},k)=\frac{1}{2}+ \frac{j-1}{q}+  \theta(n_0+mq,k)+ \frac{O(1)}{1+n_0+m}\mod \Z,
\end{equation}
for   $1\leq j\leq \frac{q}{2}$.

By the assumption that $\theta(n_0+mq,k)\in(\frac{1}{2q}-\delta,\frac{1}{2q}+\delta)$,
one has
\begin{equation}\label{G4equ5dec8ap34}
  \frac{1}{|\sin\pi \theta(n_0+mq+p_j^{\pm},k)|}=O(1),
\end{equation}
for   $1\leq j\leq \frac{q}{2}$.

  By Lemma \ref{Leapr31}, one has for $j=0,1,2,\cdots,q-1$,
  \begin{equation*}
    \theta(n_0+mq+j+1,k) \;\;\;\;\;\;\;\;\;\;\;\;\;\;\;\;\;\;\;\;\;\;\;\;\;\;\;\;\;\;\;\;\;\;\;\;\;\;\;\;\;\;\;\;\;\;\;\;\;\;\;\;\;\;\;\;\;\;\;\;\;\;\;\;\;\;\;\;
    \;\;\;\;\;\;\;\;\;\;\;\;\;\;\;\;\;\;\;\;\;\;\;\;\;\;\;\;\;\;\;\;\;\;\;\;
  \end{equation*}
  \begin{eqnarray}
     &=& k+  \theta(n_0+mq+j,k)+ \sin^2\pi  \theta(n_0+mq+j,k)\frac{V(n_0+mq+j)}{ \pi\sin \pi k}+\frac{O(1)}{1+n_0+m^2} \nonumber\\
      &=& k+  \theta(n_0+mq+j,k)+ \sin^2\pi  (\theta(n_0+mq,k)+jk)\frac{V(n_0+mq+j)}{ \pi\sin \pi k}+\frac{O(1)}{1+n_0+m^2}, \label{G4equ5}
  \end{eqnarray}
  where the last inequality holds by \eqref{G4equ5dec8ap31}.

Thus, one has
 \begin{equation*}
    \theta(n_0+mq+q,k) \;\;\;\;\;\;\;\;\;\;\;\;\;\;\;\;\;\;\;\;\;\;\;\;\;\;\;\;\;\;\;\;\;\;\;\;\;\;\;\;\;\;\;\;\;\;\;\;\;\;\;\;\;\;\;\;\;\;\;\;\;\;\;\;\;\;\;\;
    \;\;\;\;\;\;\;\;\;\;\;\;\;\;\;\;\;\;\;\;\;\;\;\;\;\;\;\;\;\;\;\;\;\;\;\;
  \end{equation*}
\begin{equation}\label{G4equ5dec8}
      = qk+  \theta(n_0+mq,k)+ \sum _{j=0}^{q-1}\sin^2\pi  (\theta(n_0+mq,k)+j k)\frac{V(n_0+mq+j)}{\pi \sin \pi k}+\frac{O(1)}{1+n_0+m^2}.
\end{equation}

By \eqref{G4equ5dec8}, in order to guarantee \eqref{Gequ20}, we only need
to  construct $a_{m,j}^{\pm}$ such that
\begin{equation}\label{Gapr6algebraic}
  \sum _{j=0}^{q-1}\sin^2\pi  (\theta(n_0+mq,k)+j k) V(n_0+mq+j)=0.
\end{equation}
 Let $ a_{m,j}^+=a$ for all $j=1,2,\cdots,\frac{q}{2}$ and $ a_{m,j}^-=a$ for all $j=1,2,\cdots,\frac{q}{2}-1$.
 By \eqref{G4equ6} and \eqref{G4equ7},
 it suffices to determine $a_{m,\frac{q}{2}}^-$ such that
 \begin{equation*}
  a\sum_{j=1}^{\frac{q}{2}} \sin^2\pi  (\theta(n_0+mq,k)+p_j^+ k)\;\;\;\;\;\;\;\;\;\;\;\;\;\;\;\;\;\;\;\;\;\;\;\;\;\;\;\;\;\;\;\;\;\;\;\;\;\;\;\;\;\;\;\;\;\;\;\;\;\;\;\;\;\;\;\;\;\;\;\;\;\;\;\;\;\;
 \end{equation*}
  \begin{equation}\label{Gequ21aarp}
   -a\sum_{j=1}^{\frac{q}{2}-1} \sin^2\pi  (\theta(n_0+mq,k)+p_j^- k)-a_{m,\frac{q}{2}}^-\sin^2\pi  (\theta(n_0+mq,k)+\frac{q-1}{q})=0.
  \end{equation}
  By the definition of $p_j^{\pm}$ (\eqref{Gdec8n1} and \eqref{Gdec8n2}) and $k=\frac{p}{q}$, it suffices to guarantee that

\begin{equation*}
  a\sum_{j=1}^{\frac{q}{2}} \sin^2\pi  (\theta(n_0+mq,k)+\frac{j-1}{q})\;\;\;\;\;\;\;\;\;\;\;\;\;\;\;\;\;\;\;\;\;\;\;\;\;\;\;\;\;\;\;\;\;\;\;\;\;\;\;\;\;\;\;\;\;\;\;\;\;\;\;\;\;\;\;\;\;\;\;\;\;\;\;\;\;\;
 \end{equation*}
  \begin{equation}\label{Gequ21}
   -a\sum_{j=1}^{\frac{q}{2}-1} \sin^2\pi  (\theta(n_0+mq,k)+\frac{1}{2}+\frac{j-1}{q})-a_{m,\frac{q}{2}}^-\sin^2\pi  (\theta(n_0+mq,k)+\frac{q-1}{q})=0.
  \end{equation}
  Direct computation implies that
  \begin{equation}\label{Gdec9551}
   \sum_{j=1}^{j=\frac{q}{2}} \sin^2\pi (\frac{1}{2q}+\frac{j-1}{q})= \sum_{j=1}^{j=\frac{q}{2}} \sin^2\pi (\frac{1}{2q}+\frac{1}{2}+\frac{j-1}{q}).
  \end{equation}
  This shows that  if $ \theta(n_0+mq,k)=\frac{1}{2q}$, \eqref{Gequ21} holds for $a_{m,\frac{q}{2}}^-=a$.

  In our case, $\theta(n_0+mq,k)\in(\frac{1}{2q}-\delta,\frac{1}{2q}+\delta)$. Then there exist $\tilde{\delta}>0$ (small) and  $a_{m,\frac{q}{2}}^-$ such that
  $ | a_{m,\frac{q}{2}}^--a|\leq \tilde{\delta}$ and  \eqref{Gequ21} holds.

\end{proof}

\section{Proof of Theorems \ref{thm2}, \ref{thm3}, \ref{thm2odd}, \ref{thm3odd}, Corollaries \ref{cor2} and \ref{cordec} }\label{embedone}
\begin{proof}[ \bf Proof of Theorems \ref{thm2} and   \ref{thm3}  for $q=0$]
Solve the following equation with initial condition $\theta(0)$,
\begin{equation}\label{PrufTodd}
    \cot (\pi \theta(n+1,k)-\pi k)=\cot \pi \theta(n,k)-\frac{V(n)}{\sin \pi k},
\end{equation}
with
\begin{equation*}
    V(n)=\frac{a}{1+n} \text{sgn}(\sin 2\pi \theta(n,k)),
\end{equation*}
where sgn$(\cdot)$ is the sign function.
Thus equation \eqref{PrufR} becomes
\begin{equation}\label{PrufRodd}
    \frac{R(n+1,k)^2}{R(n,k)^2}=1-\frac{a}{\sin \pi k}\frac{|\sin 2\pi \theta(n,k)|}{1+n}+\frac{O(1)}{1+n^2}.
\end{equation}
Applying   \eqref{Gabsence1} and following the proof of \eqref{Gequ1}, we have
\begin{equation}
    \ln R(n+1,k)^2\leq  C(k,n_0,a)  -\frac{a}{\sin \pi k} (\frac{2}{\pi}-\varepsilon) \ln n.\label{Gequ1oddap3}
\end{equation}
Suppose $E$ and $a$ satisfy the assumption for $q=0$ in Theorems \ref{thm2} and   \ref{thm3}. Then
we have
\begin{equation}\label{Gequ14}
  \frac{a}{\sin \pi k} (\frac{2}{\pi}-\varepsilon) >1
\end{equation}
for small $\varepsilon>0$.
By \eqref{Gequ1oddap3}, we obtain that  $R(n,k)$ is in $\ell^2(\N)$.
By changing the initial condition $\theta(0)$, we can make the $\ell^2(\N)$ solution $u$ satisfy the given boundary condition \eqref{Gbc}.
 We finish the proof for $q=0$.
\end{proof}

In the following arguments, we will  continue to use the idea ``making  the $\ell^2(\N)$ solution $u$ satisfy the given boundary condition \eqref{Gbc} by changing the initial condition $\theta(0)$".
In order to avoid the repetition, sometimes  we omit the details.

\begin{proof}[\bf Proof of Theorems \ref{thm2} and   \ref{thm3} for $q\geq 2$ ]

Fix $E\in S_q $ with even $q\geq 2$.

By Theorem \ref{Ledec9} and induction, we can prove that
 there exist $a_{j,m}^{\pm}$, $j=1,2,\cdots,\frac{q}{2}$, $m\geq 0$,  and $\tilde{\delta}$  such that the following statements hold,
 \begin{itemize}
   \item  for all $j=1,2,\cdots,\frac{q}{2}$,
  \begin{equation}\label{Gdec91}
  | a_{j,m}^{\pm}-a|\leq \tilde{\delta}.
  \end{equation}
   \item
   \begin{equation}\label{Gdec92}
    V(n_0+mq+p_j^{\pm})= \pm \frac{a_{j,m}^{\pm}}{1+n_0+mq}.
  \end{equation}
   \item
   \begin{eqnarray}
    \theta(n_0+mq,k) &=& mqk+\theta(n_0,k)+ O(1)\sum_{j=1}^{m}\frac{1}{1+n_0+j^2} \nonumber \\
      &=& \frac{1}{2q} +O(1)\sum_{j=1}^{m}\frac{1}{1+n_0+j^2}\mod\Z \nonumber \\
       &=& \frac{1}{2q} + \frac{O(1)}{1+n_0}\mod\Z.\label{Gdec93}
   \end{eqnarray}
 \end{itemize}

By \eqref{Gdec93} and \eqref{Gap3}, one has    for any $0\leq j\leq q-1$,
\begin{equation}\label{Gdec94}
    \sin \pi (\theta(n_0+mq+j,k) -\frac{1}{2q}- kj)=\frac{O(1)}{1+n_0}.
\end{equation}
By \eqref{Gdec93}, \eqref{Gdec94} and \eqref{GAqevendec9}, one has
\begin{equation}\label{GAqevendec9apr4}
      \frac{1}{q}\sum_{j=0}^{q-1} |\sin \pi \theta(n_0+mq+j,k)|=A_q+\frac{O(1)}{1+n_0}.
 \end{equation}

By \eqref{G4equ2}, \eqref{G4equ3}, \eqref{Gdec92}, and \eqref{Gdec94}, we have for any $0\leq j\leq q-1$,
\begin{equation*}
    \frac{V(n_0+mq+j)}{\sin \pi k}\sin 2\pi \theta(n_0+mq+j,k)>0.
\end{equation*}

By \eqref{Gdec91} and \eqref{GAqevendec9apr4} (letting $n_0$ be large), we have
\begin{equation}\label{G4equ12}
    \frac{1}{q}\sum_{j=0}^{q-1}V(n_0+mq +j)\sin 2\pi \theta(n_0+mq+j,k)\geq \frac{(a-\tilde{\delta})}{1+n_0+mq}(A_q-\varepsilon).
\end{equation}
 By \eqref{PrufR} and \eqref{G4equ12},
we have for $m\geq 1$,
\begin{equation*}
  \ln R(n_0+mq,k)^2\;\;\;\;\;\;\;\;\;\;\;\;\;\;\;\;\;\;\;\;\;\;\;\;\;\;\;\;\;\;\;\;\;\;\;\;\;\;\;\;\;\;\;\;\;\;\;\;\;\;\;\;\;\;\;\;\;\;\;\;\;\;\;\;\;\;\;\;\;\;\;\;\;\;\;\;\;\;\;
  \;\;\;\;\;\;\;\;\;\;\;\;\;\;\;\;\;\;\;\;\;\;\;\;\;\;\;\;\;\;\;\;\;\;\;\;\;\;\;\;
\end{equation*}
\begin{eqnarray}
      &\leq & \ln R(n_0,k)^2-\frac{1}{\sin\pi k}\sum_{j=0}^{m-1} \sum_{i=0}^{q-1}V(n_0+jq +i)\sin( 2\pi \theta(n_0+jq+i,k)) +\sum_{j=1}^{m} \frac{O(1)}{n_0+j^2} \nonumber\\
   &\leq & C(k,n_0,a) -\frac{1}{\sin \pi k}\sum_{j=0}^{m-1} q\frac{a-\tilde{\delta}}{n_0+jq}(A_q-\varepsilon) \nonumber\\
&\leq & C(k,n_0,a) - \frac{(a-\tilde{\delta})(A_q-\varepsilon)}{\sin \pi k} \ln m \nonumber\\
&\leq & C(k,n_0,a) - \frac{(a-\tilde{\delta})(A_q-\varepsilon)}{\sin \pi k} \ln (n_0+mq) .\label{Gequ13}
\end{eqnarray}
Suppose $E$ and $a$ satisfy the assumption for $q\geq 2$ in Theorems \ref{thm2} and   \ref{thm3}. Then
we have
\begin{equation}\label{Gequ14apr4}
  \frac{(a-\tilde{\delta})(A_q-\varepsilon)}{\sin \pi k} >1
\end{equation}
for small $\tilde{\delta},\varepsilon>0$.
By \eqref{Gequ13} and \eqref{Gequ14apr4},
one has
\begin{equation*}
   \sum_{m} R(n_0+mq,k)^2<\infty.
\end{equation*}
This implies (using \eqref{PrufR}),
\begin{equation*}
   \sum_{n} R(n,k)^2<\infty.
\end{equation*}
By \eqref{L2} and \eqref{L21}, we have
$u(n,k)$ is   in  $\ell^2(\N)$. We finish the proof for $q\geq 2$.

\end{proof}

\begin{proof}[\bf Proof of Theorems \ref{thm2odd} and   \ref{thm3odd}]
Solve the following equation,
\begin{equation}\label{PrufTodd}
    \cot (\pi \theta(n+1,k)-\pi k)=\cot \pi \theta(n,k)-\frac{V(n)}{\sin \pi k},
\end{equation}
with
\begin{equation*}
    V(n)=\frac{a}{1+n} \text{sgn}(\sin 2\pi \theta(n,k)),
\end{equation*}
where sgn$(\cdot)$ is the sign function.
Thus equation \eqref{PrufR} becomes
\begin{equation}\label{PrufRodddec}
    \frac{R(n+1,k)^2}{R(n,k)^2}=1-\frac{a}{\sin \pi k}\frac{|\sin 2\pi \theta(n,k)|}{1+n}+\frac{O(1)}{1+n^2}.
\end{equation}
Applying   \eqref{Gabsence2dec10} and  \eqref{PrufR}, we have
\begin{equation*}
    \ln R(n+q,k)^2\leq  \ln R(n,k)^2  -q(B_q-\varepsilon)\frac{a}{\sin \pi k} \frac{1}{1+n}.
\end{equation*}
This implies
\begin{eqnarray}
   \ln R(n_0+mq,k)^2 &\leq& C(n_0,k,a) -(B_q-\varepsilon)\frac{a}{\sin \pi k} \sum _{j=0}^{m-1}\frac{q}{1+n_0+jq}\nonumber\\
   &=&  C(n_0,k,a) - (B_q-\varepsilon)\frac{a}{\sin \pi k}\ln(n_0+mq) .\label{Gequ1odd}
\end{eqnarray}

Suppose $E$ and $a$ satisfy the assumption for odd $q$ in   Theorems \ref{thm2odd} and   \ref{thm3odd}. Then
we have
\begin{equation}\label{Gequ14}
 (B_q-\varepsilon)\frac{a}{\sin \pi k}  >1
\end{equation}
for small $\varepsilon>0$.
By \eqref{Gequ1odd}, we obtain   $R(n,k)$ is in $\ell^2(\N)$. We finish the proof for odd $q\geq 3$.
\end{proof}

\begin{proof}[\bf Proof of Corollary \ref{cor2}]
Let us consider $q=2$. Then $ A_2=1$ and $S_2=\{0\}$.
Now the Corollary follows Theorem \ref{thm2}.
\end{proof}

\begin{proof}[\bf Proof of Corollary \ref{cordec}]

By \eqref{Gasy} and the fact that $a<\frac{\pi}{2}$, we have  for large $q$,
\begin{equation*}
  \frac{1}{A_q}>a,
\end{equation*}
and
\begin{equation*}
(-2\sqrt{1-\frac{4a^2}{\pi^2}}+\epsilon,2\sqrt{1-\frac{4a^2}{\pi^2}}-\epsilon)\subset (-2\sqrt{1-a^2A_q^2},2\sqrt{1-a^2A_q^2}).
\end{equation*}
Theorem \ref{thm1} implies that there are no eigenvalues in $(-2\sqrt{1-\frac{4a^2}{\pi^2}}+\epsilon,2\sqrt{1-\frac{4a^2}{\pi^2}}-\epsilon)\cap S_q$ for $q=0$ and large $q$.
Now Corollary \ref{cordec} follows.
\end{proof}
\section{Proof of Theorem \ref{thm6}}\label{SecAsym}
 \begin{lemma}\cite[Lemma 4.4]{KLS}\label{Leapr7}
  Let $\{e_i\}_{i=1}^N$ be a set of unit vector in a Hilbert space $\mathcal{H}$  so that
  \begin{equation*}
    \alpha=N\sup_{i\neq j}| \langle e_i,e_j\rangle|<1.
  \end{equation*}
  Then
  \begin{equation}\label{Gapr71}
    \sum_{i=1}^N|\langle g,e_i\rangle|^2\leq (1+\alpha)||g||^2.
  \end{equation}
  \end{lemma}
  \begin{lemma}\cite[Lemma 4.2]{LDJ}\label{Lcon1}
Suppose $V(n)=\frac{O(1)}{1+n}$. Let $E_1,E_2\in (-2,2)$ be such that $k(E_1)\neq k(E_2)$ and $k(E_i)+k(E_j)\neq 1$ for $i,j=1,2$. Then  for any $\varepsilon>0$, there exist $D(E_1,E_2,\varepsilon)$ and $D(E_1,\varepsilon)$  such that
\begin{equation}\label{Gcons4}
  \left|  \sum _{t=1}^n \frac{\cos 4 \theta(t,E_1)}{1+t}\right|\leq D(E_1,\varepsilon)+ {\varepsilon}\ln n,
\end{equation}
and
\begin{equation}\label{Gcons5}
    \left|\sum_{t=1}^n\frac{\sin 2 \theta(t,E_1) \sin 2 \theta(t,E_2)}{1+t}\right|\leq D(E_1,E_2,\varepsilon)+{\varepsilon}\ln  n.
\end{equation}

\end{lemma}
\begin{proof}[\bf Proof of Theorem \ref{thm6}]
By the assumption of Theorem \ref{thm6}, for any $M>a$, we have
\begin{equation*}
 | V(n)|\leq \frac{M}{1+n}
\end{equation*}
for large $n$.
By shifting the operator, we can assume
\begin{equation}\label{Gapr76}
  | V(n)|\leq \frac{M}{1+n}
\end{equation}
for all $n>0$.

Fix $\epsilon>0$, which is small.

  Suppose  $E_1,E_2,\cdots E_N\in(0,2)$ so that  \eqref{Gdis} has an $\ell^2(\N)$ solution for each $E_i$, $i=1,2,\cdots,N$. Let $k_i=k(E_i)$ for $i=1,2,\cdots,N$. It leads that
  \begin{equation*}
    \sum_{i=1}^NR(n,k_i)\in \ell^2(\N),
  \end{equation*}
  and  then there exists $B_j\to \infty$ such that
  \begin{equation}\label{Gapr72}
    R (B_j,k_i) \leq  B_j^{-\frac{1}{2}},
  \end{equation}
   for all $i=1,2,\cdots,N$.

   Notice that we   have
   \begin{equation*}
    R(n,k_i)=O(1)
   \end{equation*}
    for all $i=1,2,\cdots,N$.

   By \eqref{PrufR}, one has
\begin{equation}\label{GPrufRmar14}
   \ln R(n+1,k)^2 - \ln R(n,k)^2=-\frac{V(n)}{\sin \pi k}\sin 2\pi \theta(n,k)+\frac{O(1)}{1+n^2}.
\end{equation}
  By \eqref{Gapr72} and \eqref{GPrufRmar14}, we have
  \begin{equation}\label{Gapr75}
\sum_{n=1}^{B_j} V(n)\sin 2\pi \theta(n,k) \geq (\sin \pi k_i)B_j+O(1),
  \end{equation}
   for all $i=1,2,\cdots,N$.

We consider the Hilbert spaces
\begin{equation*}
  \mathcal{H}_j=\{u\in \R^{B_j}:\sum_{n=1}^{B_j}|u(n)|^2(1+n)<\infty\}
\end{equation*}
with the inner product
\begin{equation*}
    \langle u,v \rangle=\sum_{n=1}^{B_j} u(n)v(n)(1+n).
\end{equation*}

In  $\mathcal{H}_j$, by \eqref{Gapr76} we have
\begin{equation}\label{Gapr77}
  ||V||_{ \mathcal{H}_j}^2\leq M^2\log (1+B_j).
\end{equation}
Let
\begin{equation*}
  e^j_{i}(n)=\frac{1}{\sqrt{A_i^j}}\frac{\sin 2\theta(n,k_i)}{1+n}\chi_{[0,B_j]}(n),
\end{equation*}
where $A_i^j$ is chosen so that $e_i^j$ is a unit vector in $\mathcal{H}_j$.
We have the following estimate,
\begin{eqnarray}
  A_i^j &=& \sum_{n=1}^{B_j}\frac{\sin^2 2\theta(n,k_i)}{1+n} \nonumber\\
   &=&\sum_{n=1}^{B_j}\frac{1}{2(1+n)}- \sum_{n=1}^{B_j}\frac{\cos  4\theta(n,k_i)}{2(1+n)}\nonumber\\
   &=& \frac{1}{2}\log B_j- \sum_{n=1}^{B_j}\frac{\cos  4\theta(n,k_i)}{2(1+n)}+O(1),\label{Gapr791}
\end{eqnarray}
Since $E_i$ is positive ($k_i\in (0,\frac{1}{2})$) for all $i=1,2,\cdots,N$, one has
  \begin{equation*}
    k_i+k_{\tilde{i}}\neq 1,
  \end{equation*}
  for all $1\leq i,\tilde{i}\leq N$.

By \eqref{Gcons4}, one has
\begin{equation}\label{Gcons4Gapr79}
  \left|  \sum _{n=1}^{B_j} \frac{\cos 4 \theta(n,k_i)}{1+n}\right|\leq O(1)+ \epsilon\ln B_j,
\end{equation}
for all $i=1,2,\cdots,N$.

By \eqref{Gapr791} and \eqref{Gcons4Gapr79}, we have
\begin{equation}\label{Gapr79}
  (\frac{1}{2}-\epsilon) \log B_j+O(1)  \leq A_i^j\leq (\frac{1}{2}+\epsilon)\log B_j+O(1).
\end{equation}

By  \eqref{Gcons5}, we have for $i\neq \tilde{i}$,
\begin{equation}\label{Gapr101}
-\epsilon\log B_j+O(1) \leq \sum_{n=1}^{B_j}  \frac{\sin  2\theta(n,k_i)\sin  2\theta(n,k_{\tilde{i}})}{1+n}\leq \epsilon\log B_j+O(1).
\end{equation}
By \eqref{Gapr79} and \eqref{Gapr101}, one has
\begin{equation*}
-4\epsilon +\frac{O(1)}{\log B_j}  \leq \langle e_i^j,e_{ \tilde{i}}^j \rangle\leq 4\epsilon +\frac{O(1)}{\log B_j},
\end{equation*}
for  all $1\leq i,\tilde{i}\leq N$ and $i\neq \tilde{i}$.
It implies  for large $j$,
\begin{equation}\label{Gapr78}
-5\epsilon   \leq \langle e_i^j,e_{ \tilde{i}}^j \rangle\leq 5\epsilon,
\end{equation}
for  all $1\leq i,\tilde{i}\leq N$ and $i\neq \tilde{i}$.

By \eqref{Gapr79} and \eqref{Gapr75}
\begin{equation}\label{Gapr710}
 \langle V,e^j_i \rangle_{\mathcal{H}_j}\geq \sqrt{2}(1-2\epsilon)\sin \pi k_i \sqrt{\log B_j},
\end{equation}
for large $j$.
By \eqref{Gapr71} and \eqref{Gapr78}, one has
\begin{equation}\label{Gapr711}
\sum_{i=1}^N  |\langle V,e^j_i\rangle_{\mathcal{H}_j}|^2\leq (1+10N\epsilon)||V||_{\mathcal{H}_j}.
\end{equation}
By \eqref{Gapr710}, \eqref{Gapr711} and \eqref{Gapr77},
we have
\begin{equation*}
  \sum_{i=1}^N 2(1-2\epsilon)^2\sin ^2(\pi k_i)  \log B_j\leq (1+10N\epsilon) M^2 \log  B_j+O(1).
\end{equation*}
Let $j\to \infty$ and then $\epsilon\to 0$, we get
\begin{equation*}
  \sum_{i=1}^N 2\sin ^2(\pi k_i)   \leq  M^2 ,
\end{equation*}
for any $M>a$. This implies
\begin{equation*}
  \sum_{i=1}^N 2\sin ^2(\pi k_i)   \leq  a^2 .
\end{equation*}
By the fact that $E_i=2\cos \pi k_i$, one has
\begin{equation}\label{Gapr102}
  \sum_{i=1}^N (4-E_i^2)   \leq  2a^2 .
\end{equation}
It yields that $H_0+V$ has at most countable  eigen-solutions corresponding to positive energies.
Let $N\to \infty$ in \eqref{Gapr102}, we have
\begin{equation}\label{Gapr103}
  \sum_{E_i\in (0,2)\cap P}^{\infty} (4-E_i^2)   \leq  2a^2 .
\end{equation}
Similarly,  we can prove $H_0+V$ has at most countable   negative eigenvalues in $(-2,0)$ with the same  upper bound  in \eqref{Gapr103}.
Thus
\begin{equation}\label{Gapr104}
  \sum_{E_i \in P,E_i\neq 0}^{\infty} (4-E_i^2)   \leq  4a^2 .
\end{equation}
We may add 4 in the bound of \eqref{Gapr104} if   $Hu=0$ has an $\ell^2(\N)$ solution.
However, if $ a<1$, by Theorem \ref{thm1}, $Hu=0$ can not have an $\ell^2(\N)$ solution.
We finish the proof.
\end{proof}
\section{Proof of Theorems \ref{thm4} and \ref{thm5}}\label{Smany}
Suppose $u(n,E)$ is a solution of $Hu=H_0u+Vu=Eu$.
Let
\begin{equation*}
    \tilde{R}(n,E) =\sqrt{u(n-1,E)^2+u(n,E)^2}.
\end{equation*}

%
%
%
%
%
%
%

    \begin{lemma}(\cite[Theorem 3.1]{MR2945209})\label{Keylemma}
    Let
    \begin{equation*}
          {V}(n)=\frac{M\sin(2\pi k n+\phi)}{n},
    \end{equation*}
    with $k\neq \frac{1}{2}$ and $H=H_0+V$.
     Then for every $E\in(-2,2)$ there exists a base $u^+(n,E)$ and $u^-(n,E)$ of $Hu=Eu$ with the following asymptotics.
    \\
    Case 1. For $E=2\cos\pi k$

    \begin{equation*}
        u^+(n,E)=n^{\frac{M}{4\sin\pi k}}(\cos(\pi k n+\phi/2)+o(1)),
    \end{equation*}
    and
    \begin{equation}\label{Gconsnew1}
        u^-(n,E)=n^{-\frac{M}{4\sin\pi k}}(\sin(\pi k n+\phi/2)+o(1)).
    \end{equation}
    Case 2. For $E=-2\cos\pi k$
    \begin{equation*}
    \begin{array}{l}
        u^+(n,E)=(-1)^nn^{\frac{M}{4\sin\pi k}}(\sin(\pi k n+\phi/2)+o(1)),
        \\
        u^-(n,E)=(-1)^nn^{-\frac{M}{4\sin\pi k}}(\cos(\pi n+\phi/2)+o(1)).
    \end{array}
    \end{equation*}
   Case 3. For $E=2\cos\pi \hat{k}\in(-2;2)\backslash\{\pm2\cos\pi k\}$
    \begin{equation*}
    \begin{array}{l}
        u^+(n,E)=\exp(i\hat{k}n)+o(1),
        \\
        u^-(n,E)=\exp(-i\hat{k} n)+o(1).
    \end{array}
    \end{equation*}
    \end{lemma}
\begin{proposition}\label{Twocase}
Let  $E\neq 0$ and $ A=\{\tilde{E}_j\}_{j=1}^m$ be in $(-2,2)$ such that $E\notin (-A)\cup A$.  Suppose $E$ and $\{\tilde{E}_j\}_{j=1}^m$ are different.
Suppose  $\theta_0\in[0,\pi)$. Let $n_1>n_0>b$.
Then there exist constants $K(E, A)$, $C(E, A)$ (independent of $b, n_0$ and $n_1$),     and potential $\widetilde V(n,E,A,n_0,n_1,b,\theta_0)$  such that  for $n_0-b>K(E,A)$ the following holds:

   \begin{description}
     \item[Potential]   for $n_0\leq n \leq n_1$, ${\rm supp}(\widetilde V)\subset(n_0,n_1)$,  and
     \begin{equation}\label{thm141}
        |\widetilde V(n,E,A,n_0,n_1,b,\theta_0)|\leq \frac{C(E, A)}{n-b}.
     \end{equation}

     \item[Solution for $E$] the solution of $(H_0+\widetilde V)u=Eu$ with boundary condition $\frac{u(n_0,E)}{u(n_0-1,E)}=\tan \theta_0$ satisfies
     \begin{equation}\label{thm142}
        \tilde R(n_1,E)\leq  C(E,A)(\frac{n_1-b}{n_0-b})^{-100}\tilde R(n_0,E)
     \end{equation}
     and  for $n_0<n<n_1$,
      \begin{equation}\label{thm143}
       \tilde R(n,E)\leq C(E,A) \tilde R(n_0,E).
     \end{equation}
      \item[Solution for $\tilde{E}_j$] any  solution of $(H_0+\widetilde V)u=\tilde{E}_ju$ satisfies
      for $n_0<n\leq n_1$,
      \begin{equation}\label{thm144}
       \tilde R(n,\tilde{E}_j)\leq  C(E,A) \tilde R(n_0,\tilde{E}_j).
     \end{equation}
   \end{description}
   \begin{proof}

   For simplicity, denote by $K=K(E,A)$,  $C=C(E,A)$ and $k=k(E)$.
   Let $E=2\cos\pi k$.
   By the assumption that $E\neq 0$, one has $k\neq \frac{1}{2}$.
   By the assumption, we have $\tilde{E}_j\neq \pm E$.

   By shifting the operator $b$ unit, we only need to  consider $b=0$.
   For $n\geq n_0$, define
   \begin{equation}\label{GdefV}
    \tilde{V}(n)=\frac{400}{\sin \pi k(E)}\frac{\sin(2\pi k n+\phi_E)}{n},
   \end{equation}
   where $\phi_E$ will be determined later.

   By Case 1 of Lemma \ref{Keylemma}, one of the solution of $Hu=H_0u+\tilde{V}u=Eu$ satisfies \eqref{Gconsnew1}.
   By adapting $\phi_E$ in \eqref{GdefV}, we can make sure that the solution of $Hu=H_0u+\tilde{V}u=Eu$ with boundary condition $\frac{u(n_0,E)}{u(n_0-1,E)}=\tan \theta_0 $
   satisfies \eqref{Gconsnew1}.
   Thus (choosing $C$ large enough in \eqref{GdefV}), one has
   \begin{equation}\label{thm142new}
        \tilde R(n_1,E)\leq  C (\frac{n_1}{n_0})^{-100}\tilde R(n_0,E)
     \end{equation}
     and  for $n_0<n<n_1$,
      \begin{equation}\label{thm143new}
       \tilde R(n,E)\leq  C\tilde R(n_0,E).
     \end{equation}
     Those prove \eqref{thm142} and \eqref{thm143}.

\eqref{thm144} follows from Case 3 of Lemma \ref{Keylemma}.

\end{proof}

  \end{proposition}

\begin{proof}[\bf Proof of Theorems \ref{thm4} and \ref{thm5}]
Once we have Proposition \ref{Twocase} at hand,
 Theorems \ref{thm4} and \ref{thm5}  can be proved by the   piecewise functions gluing technics from  \cite{jl,ld}.

\end{proof}
 \section*{Acknowledgments}
 I would like to thank Svetlana Jitomirskaya for comments on earlier versions of the manuscript.
    W.L. was supported by the AMS-Simons Travel Grant 2016-2018. This research was also supported by  NSF DMS-1401204 and  NSF DMS-1700314.

\end{document}